\documentclass[10pt,a4paper]{amsart}
\usepackage{amscd}
\usepackage{amsmath,amssymb, amstext,amsthm,amsfonts}
\usepackage{enumitem}
\usepackage{xcolor}
\usepackage{comment}
\usepackage{thmtools,thm-restate}
\usepackage[utf8x]{inputenc}
\usepackage[T1]{fontenc}
\usepackage[]{hyperref}

\newtheorem{Thm}{Theorem}
\newtheorem{theorem}{Theorem}[section] 
\newtheorem{lemma}[theorem]{Lemma}     
\newtheorem{corollary}[theorem]{Corollary}
\newtheorem{Prop}[theorem]{Proposition}

\newtheorem*{corollary*}{Corollary}

\theoremstyle{definition}
\newtheorem{defi}[theorem]{Definition}
\newtheorem{fact}[theorem]{Fact}

\newcommand{\al} {\alpha}       

\newcommand{\be} {\beta}

\newcommand{\de} {\delta}

\newcommand{\vep}{\varepsilon}
\newcommand{\vphi}{\varphi}

\newcommand{\la} {\lambda}

\newcommand{\si} {\sigma}

\newcommand{\wh}{\widehat}
\newcommand{\wt}{\widetilde}

\newcommand{\SA}{{\mathcal A}}

\newcommand{\SM}{{\mathcal M}}

\newcommand{\SO}{{\mathcal O}}
\newcommand{\SP}{{\mathcal P}}
\newcommand{\SQ}{{\mathcal Q}}

\newcommand{\SU}{{\mathcal U}}

\newcommand{\SW}{{\mathcal W}}

\newcommand{\N}{\mathbb{N}}

\newcommand{\R}{\mathbb{R}}

\newcommand{\Z}{\mathbb{Z}}

\newcommand{\flowF}{\Phi}
\newcommand{\tmap}{\vphi^t}
\newcommand{\smap}{\vphi^s}
\newcommand{\hsmap}{\vphi^{h(s)}}

\newcommand{\Erg}{\operatorname{\SM_\Phi^e}}
\newcommand{\Mes}{\operatorname{\SM_\Phi}}
\newcommand{\fk}{\operatorname{\rho_{FK}}}
\newcommand{\pfk}{\operatorname{\rho^{1}_{FK}}}

\newcommand{\fkflow}{\operatorname{\tilde{\rho}_{FK}}}

\newcommand{\Part}{\operatorname{\mathbf{P}}}

\newcommand{\ack}{{\bf Acknowledgements. }}

\newcommand{\supp}{\operatorname{supp}}
\newcommand{\gen}{\operatorname{Gen}}

\numberwithin{equation}{section}

\begin{document}
	
	\title[A characterization of zero entropy loosely Bernoulli flows]{A characterization of zero entropy loosely Bernoulli flows via FK-pseudometric}
	
	\author{Alexandre Trilles}
	\address{Doctoral School of Exact and Natural Sciences, Jagiellonian University, ul. {\L}ojasiewicza 11, 30-348 Krak\'ow, Poland
		\& Faculty of Mathematics and Computer Science, Jagiellonian University, ul. {\L}ojasiewicza 6, 30-348 Krak\'ow, Poland.}

	\email{alexandre.trilles@doctoral.uj.edu.pl}
	
	\thanks{The author was supported by NCN Polonez Bis grant no. 2021/43/P/ST1/02885.} 
	

	\subjclass[2020]{Primary 37A35, 37A10; Secondary 37B05}
	

	\keywords{}
	
	\begin{abstract}
		We introduce the Feldman-Katok pseudometric (FK-pseudometric for short) for flows. We then provide a characterization of zero entropy loosely Bernoulli measures for continuous flows via the FK-pseudometric extending the result known for discrete-time dynamical systems. We also provide a purely topological characterization of uniquely ergodic continuous flows whose unique invariant measure is zero entropy loosely Bernoulli.
		
	\end{abstract}
	
	\maketitle
	
	\section{Introduction}
	
	The theory of dynamical systems originated historically from the study of flows generated by differential equations. Poincar\'e used the transformation known today as Poincar\'e map to reduce study of flows to the discrete-time case. There are several ways of passing from continuous to discrete-time dynamics, however this reduction is not always direct nor obvious, which justifies the study of flows.
	
	Despite the strong connection between continuous and discrete-time dynamics, it is sometimes not immediate to find direct analogies between these settings. For example, the definition of expansivity for homeomorphisms turns out to be unfit for continuous flows. Indeed, no non-trivial flow satisfies the direct translation of the discrete-time definition \cite{Bowen-Walters}. 
	Therefore, a lot of work has been done to define expansivity for flows so that it behaves similarly to the discrete-time notion. The first definition of expansivity that appeared for flows (due to Bowen and Walters \cite{Bowen-Walters}), for example, preserves certain results known for homeomorphisms, but it does not work well with flows with singularities, where the so-called $k^*$-expansivity (due to Komuro \cite{Komuro}) is more suitable. 
	In summary, in order to find an analogue definition, several proposals were made, and each one of them preserves different results known for homeomorphisms on different classes of flows (see Section 1.7 in \cite{Hasselblatt-Fisher}). 
	In this work we propose a definition of a pseudometric for flows corresponding to the Feldman-Katok pseudometric introduced in \cite{Kwietniak-Lacka} for discrete-time systems and we show that it preserves results obtained in \cite{KG-R}.
	
	Two  topologically conjugated discrete-time systems have the same dynamics from the topological perspective. Unfortunately, in the continuous-time case topological conjugacy seems to be too restrictive. A very small time-change of a continuous flow might not be conjugated to the initial one. In this setting a weaker form of equivalence called (topological) orbit equivalence seems to be more natural. Orbit equivalence guarantees similarity between a flow and its time-changes. However, it fails to preserve certain topological properties and quantities such as mixing and entropy. 
	The equivalence relations can also be considered from the measure-theoretical viewpoint and equivalent systems are supposed to share the same dynamics from the perspective of ergodic theory. Unfortunately, the measurable orbit equivalence becomes too weak. Dye, in \cite{Dye_1} and \cite{Dye_2}, showed that any two ergodic automorphisms are measurably orbit equivalent and his result can be extended to the context of ergodic flows by using cross-sections.
	
	An intermediate equivalence relation, originally called weak isomorphism, was introduced by Kakutani \cite{Kakutani}. Nowadays it is known as Kakutani equivalence or monotone equivalence. Two flows are Kakutani equivalent if there is an isomorphism between one flow and an $L^1$ time-change of the other (see Section \ref{erg_theory} for the precise definition). Using Abramov's formula we see that this equivalence relation preserves complexity of systems, meaning that classes of zero, positive but finite and infinite entropy are preserved under Kakutani equivalence, but it is important to emphasize that entropy is not an invariant under Kakutani equivalence.
	
	In the opposite direction, a remarkable result on the isomorphism problem, due to Ornstein (see \cite{Ornstein_isomorphism}, \cite{Ornstein_isomorphism2}, \cite{Ornstein_isomorphism3}) says that Bernoulli shifts with the same entropy are isomorphic. In 1976, studying the isomorphism problem in ergodic theory Feldman \cite{Feldman_LB} introduced the definition of loosely Bernoulli automorphisms inspired by Ornstein's definition of very weak Bernoulli from \cite{Ornstein_isomorphism2}. 
	One first defines the property for finite partitions and then for automorphisms. While the notion of very weak Bernoulli partition is based on the normalized Hamming distance $\bar{d}_n$ between strings of length $n$, the notion of loose Bernoullicity is based on the weaker edit metric $\bar{f}_n$. These distances are related to dynamical pseudometrics for discrete-time systems, the first one called Besicovitch pseudometric and the second one the Feldman-Katok pseudometric (see \cite{Besicovitch_KLO}, \cite{Kwietniak-Lacka} for details).
	
	Surprisingly, a loosely Bernoulli system can have zero entropy. Indeed, Kronecker systems are loosely Bernoulli and even more, a system with zero entropy is loosely Bernoulli if and only if it is Kakutani equivalent to a Kronecker system. The study of the positive and the zero entropy case is done separately and we will focus on loosely Bernoulli systems with zero entropy. Following Ratner's suggestion we call them \emph{loosely Kronecker} (Katok called these systems standard \cite{Katok_discrete_spec}). 
	
	Roughly speaking, Kakutani \cite{Kakutani} conjectured that all zero entropy systems were loosely Kronecker. The conjecture turned out to be false and the first example of a zero entropy but not loosely Kronecker system was given by Feldman \cite{Feldman_LB} followed by results of Katok \cite{Katok_discrete_spec} and Ornstein, Rudolph and Weiss \cite{ ORW} who independently constructed uncountably many zero entropy systems that are not Kakutani equivalent one to another.
	
	In Katok's and Ornstein, Rudolph and Weiss' works, the study of loosely Kronecker flows is reduced to automorphisms. This is based on the fact that every ergodic measure-preserving flow with essentially no fixed points is isomorphic to a special flow over an ergodic transformation with an $L^1_+$ roof function (see \cite{Ambrose-Kakutani}). In the late 70's Ratner proved that horocycle flows are loosely Kronecker and that the Cartesian square of horocycle flows is not loosely Kronecker (see \cite{Ratner_horocycle}, \cite{Ratner_Cartesian}). Her proofs are again based on discrete-time results applied to cross-sections.
	
	In the early 80's there emerged results dealing directly with flows. The first one is due to Feldman \cite{Feldman_rentropy} who proved among many other things a characterization of loosely Kronecker flows. Inspired by his work, Ratner \cite{Ratner_invariants} introduced invariants of Kakutani equivalence that were used to show that different cartesian powers of horocycle flows are not Kakutani equivalent.
	
	In her last paper \cite{Ratner_2017}, Ratner presented a slightly different definition of her invariants from \cite{Ratner_invariants}. Inspired by Ratner's definition which can be seen as a continuous-time version of the $\bar{f_n}$-metric, we define the Feldman-Katok pseudometric (FK-pseudometric for short) for flows.
	
	We show that our pseudometric behaves similarly to the discrete-time one, allowing us to provide another characterization of loosely Kronecker continuous flows extending the results of Garc\'ia-Ramos and Kwietniak \cite{KG-R} from the discrete-time to the continuous-time setting. We denote the FK-pseudometric for flows by $\fkflow$.
	
	\begin{restatable}{thm}{mainthm}\label{thmA}
		Let $\flowF$ be a continuous flow on $X$ and $\mu$ be an ergodic $\flowF$-invariant measure. The measure-preserving flow $(X, \flowF, \mu)$ is loosely Kronecker if and only if there exists a Borel set $H \subseteq X$ such that $\mu(H) = 1$ and $\fkflow(x,y) = 0$ for every $x,y \in H$.
	\end{restatable}
	
	The main difference between our proof and the one for discrete-time systems by Garc\'ia-Ramos and Kwietniak is that their proof uses Katok's criterion presented in \cite{Katok_discrete_spec} while ours is based on Ratner's criterion from \cite{Ratner_2017}.
	
	Motivated by the discrete-time version of Theorem \ref{thmA}, Garc\'ia-Ramos and Kwietniak introduced the notion of \emph{topologically loosely Kronecker systems} where instead of points in a set of full measure, every pair of points is indistinguishable with respect to the FK-pseudometric, meaning $\fk(x,y)=0$ for every $x,y \in X$. We will use the same terminology for flows replacing $\fk$ with its variant for flows $\fkflow$.
	
	For both, discrete and continuous-time case, the connection between measure-preserving and continuous (topological) systems is well established. On one hand, the Krylov-Bogoliubov Theorem guarantees the existence of an invariant measure for continuous flows and maps. On the other hand, the Jewett-Krieger Theorem says that every ergodic transformation is isomorphic to a uniquely ergodic continuous map. 
	We call a continuous flow a \emph{topological model} of a measure-preserving flow if it is uniquely ergodic and isomorphic to the measure-preserving one with respect to its unique ergodic invariant measure.
	For flows, Denker and Eberlein \cite{Denker-Eberlein} obtained a result similar to the Jewett-Krieger Theorem proving that every ergodic measure-preserving flow admits a minimal topological model. 
	
	As Theorem \ref{thmA}, the following result is another continuous-time version of a result obtained by Garc\'ia-Ramos and Kwietniak in \cite{KG-R}. 
	
	\begin{restatable}{thm}{secthm}\label{thmB}
		Let $\flowF$ be a continuous flow and $\Erg(X)$ be the set of all ergodic $\flowF$-invariant measures. Then $\flowF$ is topologically loosely Kronecker if and only if $\flowF$ is uniquely ergodic and $(X,\flowF, \mu)$ is loosely Kronecker, where $\{\mu\}~=~ \Erg(X)$.
	\end{restatable}
	
	\begin{corollary*}
		A continuous flow is a topological model of a loosely Kronecker measure-preserving flow if and only if it is topologically loosely Kronecker.
	\end{corollary*}
	
	Continuous uniquely ergodic flows whose unique invariant measure is loosely Kronecker are present in the literature. For example, strictly ergodic distal flows are contained in this class. Even in the smooth setting this class of systems is non-empty since it contains horocycle flows on compact surfaces of constant negative curvature, however Kanigowski, Vinhage and Wei in \cite{KVW} showed that those are essentially the only loosely Kronecker unipotent flows.
	
	In general, the connection between ergodic notions and their topological counterparts is not so strong. For example, although the restriction of a continuous flow to the support of a mixing measure is topologically mixing, there are topological models of mixing measure-preserving flows that are not topologically mixing. There are also uniquely ergodic topologically mixing flows whose unique invariant measure is not mixing. In summary, different topological models of a given ergodic flow can manifest distinct behaviors from a topological perspective. Analogously to Theorem 4.5 in \cite{KG-R} for discrete-time systems, Theorem \ref{thmB} provides a purely topological characterization of continuous uniquely ergodic flows whose unique invariant measure is loosely Kronecker. Moreover, this characterization reveals that for loosely Kronecker systems there is a deep connection between the measure-theoretical notion and its topological counterpart.
	
	Here is the organization of the paper.
	In section 2 we present the basic definitions about flows and ergodic theory that appear in this work. In section 3 we recall the definition of the FK-pseudometric for the discrete case and its properties, then we introduce our definition for flows stipulating parallels between the two. In section 4 we recall Ratner's characterization of loosely Kronecker flows and prove Theorem \ref{thmA}. In section 5 we present the notion of topologically loosely Kronecker flows and prove the Theorem \ref{thmB}. In section 6 we provide examples of topologically loosely Kronecker flows and we present some consequences of Theorem \ref{thmA} and Theorem \ref{thmB}.
	
	\section{Basic definitions}
	Throughout this work by $X$ we mean a compact metrizable space and $d$ is a compatible metric. Whenever we mention measure-theoretical properties of $X$ we will consider $X$ to be a measurable space endowed with its Borel $\si$-algebra.
	The Lebesgue measure on the space of real numbers $\R$ will be denoted by $\la$. The set of non-negative real numbers will be denoted by $\R^+$.
	
	\subsection{Flows}
	A flow $\flowF$ on $X$ is a map $\flowF \colon \R \times X \to X$ such that $\flowF(0,x)=x$ and $\flowF(t+s,x) = \flowF(s, \flowF(t,x))$ for all $x \in X$ and $t,s \in \R$. If the map $\flowF$ is continuous (resp. measurable) with respect to the product topology (resp. product $\si$-algebra) on $\R \times X$ we say it is a \textbf{continuous flow} (resp. \textbf{measurable flow}). For $t \in \R$ the \textbf{time-$t$ map} is defined as $\tmap  : = \flowF(t, \cdot) \colon X \to X$. Since $\vphi^0$ is simply the identity, we will tacitly assume $t \neq 0$ when referring to time-$t$ maps.  
	
	The \textbf{forward orbit} of $x$ under $\flowF$ is  $\SO^{+}(x) :=\left\{\varphi^t(x) \in X : t \in \R^+ \right\}$  and the \textbf{backward orbit} of $x$ under $\flowF$ is $\SO^{-}(x):=\left\{\varphi^{-t}(x) \in X : t \in \R^+\right\}$. The \textbf{orbit} of $x$ under $\flowF$ is $\SO(x) := \SO^+(x) \cup \SO^-(x)$. A continuous flow is \textbf{transitive} (resp. \textbf{minimal}) if for some (resp. every) $x \in X$, the set $\SO^+(x)$ is dense in $X$.
	A \textbf{singularity} of a flow is a point $x \in X$ such that $\SO(x)=\{x\}$. In this work we will only consider flows without singularities.
	
	Let $T\colon X \to X$ be a homeomorphism, $\al \colon X \to \R^+\backslash\{0\}$ a continuous function and the map $\tau: X \times \R \to X \times \R$ given by $\tau(x,s) = (T(x), s - \al(x))$. We denote by $X^\al_T =X \times \R\big/ \sim$ the quotient space where $(x, s) \sim \tau^n(x,s)$ for every $n \in \Z$.
	
	The \textbf{special flow} over the transformation $T$ with roof function $\al$ is the flow on $X^\al_T$ induced by the time translation $f^t(x,s) = (x, s+t)$. In case $\al \equiv 1$ we call it the \textbf{suspension flow over $T$} and we denote $X^1_T = X_T$. It is possible to define a metric $d_T$ on $X_T$ induced by $d$ assuming without loss of generality that $(X,d)$ has diameter at most 1 (see \cite{Thomas_suspension_flow} for details). For $t \in [0,1)$, the restriction of $d_T$ to $X \times \{t\}$ is given for $x,y \in X$ by
	$$d_T((x,t), (y,t)) := (1-t)d(x,y) + td(T(x), T(y)).$$
	
	\subsection{Ergodic Theory}\label{erg_theory}
	By $\SM(X)$ we denote the set of all Borel probability measures on $X$. The \textbf{support} of a measure $\mu  \in \SM(X)$ is $\supp(\mu) := \{ x \in X : \mu(U)>0 \text{ for all $U \subseteq X$ open neighborhood of } x \}$. Let $\flowF$ be a measurable flow on $X$.  We say that a measure $\mu \in \SM(X)$ is \textbf{$\flowF$-invariant} if it is $\tmap$-invariant for every $t \in \R$. The set of all $\flowF$-invariant measures will be denoted by $\Mes(X)$.  
	
	We say that an invariant measure $\mu \in \Mes(X)$ is \textbf{ergodic} if for any measurable subset $A \subseteq X$ with $\tmap(A)=A$ for all $t \in \R$ either $\mu(A) = 0$ or $\mu(A)=1$. The set of ergodic measures will be denoted by $\Erg(X)$.
	
	It is well known that $\SM(X)$ endowed with the so-called weak* topology is a compact metrizable space and $\Mes(X) \subseteq \SM(X)$ a compact subspace. It is also well known that such topology is induced by the Prokhorov metric given by
	$$D_P(\mu,\nu) = \inf \{ \vep >0 : \mu (B) \leq \nu(B^\vep) + \vep \text{ for every Borel set } B \subseteq X\},$$
	where $B^\vep$ denotes the $\vep$-hull of $B$, that is, $B^\vep =\{y \in X: d(y,B)< \vep\}$.
	
	Let $\de_x$ denote the Dirac measure supported on $\{x\}$ for $x \in X$. For $t \in \R^+$ the \textbf{$t$-empirical measure of $x \in X$} (with respect to $\flowF$) will be denoted by
	$$\mu_{x,t} : = \frac{1}{t} \int_{0}^{t} \de_{\vphi^s(x)} ds.$$
	
	We say that a point $x \in X$ is \textbf{generic} for a measure $\mu \in \Mes(X)$ if $\mu_{x,t}$ converges to $\mu$ in the weak* topology as $t \to \infty$. The set of generic points of a measure $\mu$ will be denoted by $\gen(\mu)$. As a consequence of the Birkhoff Ergodic Theorem, we have $\mu(\gen(\mu))=1$ for every $\mu \in \Erg(X)$.
	
	We call a triple $(X, \flowF, \mu)$ a \textbf{measure-preserving flow} if $\flowF$ is a measurable flow on $X$ (not necessarily continuous) and $\mu \in \Mes(X)$. In the case that $\mu \in \Erg(X)$ we will simply say that $(X, \flowF, \mu)$ is an  \textbf{ergodic flow}. A measure-preserving flow $(X, \flowF, \mu)$ is called \textbf{uniquely ergodic} if $\{\mu\} = \Erg(X)$.
	
	Let $(X,\mu)$ be a Lebesgue space, $T \colon X \to X$ a measure-preserving automorphism and $\al \in L^1_+(X,\mu)$. Analogously as done before one can define $X^\al_T$ and a measurable counterpart of the special flow.  Note that that the normalization of the measure induced by $\mu \times \la$ is invariant under the special flow.
	
	\begin{defi}
		We say that two measure-preserving flows $(X, \flowF, \mu)$ and $(Y, \Psi, \nu)$ are \textbf{isomorphic} if there exists a measure-preserving isomorphism $h \colon X \to Y$ such that $(h \circ \tmap)(x) = (\psi^t \circ h)(x)$ for every  $t \in \R$ and $\mu$-almost every $x \in X$. 
		
	\end{defi}
	
	It is worth mentioning the Ambrose-Kakutani Representation Theorem which says that every measure-preserving flow is isomorphic to a special flow.
	
	\begin{defi}
		A continuous flow $\flowF$ on $X$ is said to be a \textbf{topological model} of the measure-preserving flow $(Y, \Psi, \nu)$ if it is uniquely ergodic and $(X, \flowF, \mu)$ is isomorphic to $(Y,\Psi, \nu)$, where $\{\mu\}=\Erg(X)$.
	\end{defi}
	
	\begin{defi}
		A measure-preserving flow $(X, \flowF, \mu)$ is called \textbf{Kronecker} if it is isomorphic to a special flow over an irrational rotation on the circle. 
	\end{defi}
	
	We say that a flow $\Psi$ on $X$ is a \textbf{time-change} of $\flowF$ if there exists $\al \in L^1_+(X,\mu)$ such that
	$$\psi^t (x) = \vphi^{v(t,x)}(x),$$
	where $v(t,x)$ is the solution to 
	$$\int_{0}^{v(t,x)} \al(\smap(x))ds = t.$$
	It follows that $\Psi$ preserves the probability measure $d\nu = \left(\frac{\al}{ \int_{X} \al d\mu}\right) d\mu$.
	If, in addition, the function $\al$ is continuous, we will say it is a \textbf{continuous time-change}.
	
	\begin{defi}
		We say that two measure-preserving flows are \textbf{Kakutani equivalent} if one of them is isomorphic to a time-change of the other. 		
	\end{defi}
	
	\begin{defi}
		A measure-preserving flow $(X, \flowF, \mu)$ is called \textbf{loosely Kronecker} if it is Kakutani equivalent to a Kronecker flow.
	\end{defi}	
	
	\section{Feldman-Katok pseudometric} 
	
	\subsection{Discrete Feldman-Katok pseudometric}
	
	Kwietniak and {\L}{\k a}cka \cite{Kwietniak-Lacka} defined the Feldman-Katok pseudometric (FK-pseudometric for short) in order to prove that measures obtained by the GIKN construction (see \cite{KING}) have zero entropy. The result follows from the fact that such measures are FK-limits of loosely Kronecker measures and therefore they are also loosely Kronecker.
	
	Later, the FK-pseudometric turned out to be even more connected to loosely Kronecker maps, providing a characterization of loosely Kronecker ergodic measures for continuous maps, see \cite{KG-R}.  
	
	Let us recall the pseudometric with respect to a continuous map $T \colon X \to X$. Fix $x,y \in X$, $\de>0$ and $n \in \N$. We define an \textbf{$(n,\de)$-matching} between $x$ and $y$ to be an order preserving bijection $\pi \colon D \to D'$ such that $D,D' \subseteq \{0, \ldots, n-1\}$ and for every $i \in D$ we have $d(T^i(x), T^{\pi(i)}(y))< \de$. Given an $(n,\de)$-matching $\pi \colon D \to D'$, let $|\pi|$ be the cardinality of $D$.
	
	\begin{defi}
		The \textbf{$(n,\de)$-gap} between $x$ and $y$ is given by
		$$
		\bar{f}_{n,\de}(x,y) = 1 - \frac{\max \{|\pi| : \pi \text{ is a $(n,\de)$-matching of $x$ and $y$}\}}{n}.$$
	\end{defi}
	
	\begin{defi}
		The discrete-time \textbf{$f_\de$-pseudometric} is given by
		$$\bar{f}_\de(x,y) = \limsup_{n \to \infty} \bar{f}_{n,\de}(x,y).$$
	\end{defi}
	
	Perhaps the most relevant property of the $\bar{f}_\de$-pseudometric is its invariance along orbits which will also be preserved by the FK-pseudometric.
	
	\begin{fact}
		For every $x,y \in X$ and $\de >0$ we have $ \bar{f}_{\de}(T^k(x),y)= \bar{f}_{\de}(x,y)$ for every $k \in \N$. 
	\end{fact}
	
	\begin{defi}
		The discrete-time \textbf{Feldman-Katok pseudometric} is given by
		$$\fk (x,y) = \inf \{\de>0 : \bar{f}_\de(x,y) \leq \de\}.$$
	\end{defi}
	
	\subsection{Feldman-Katok pseudometric for flows}
	Inspired by the $\bar{f}$-metric Kwietniak and {\L}{\k a}cka introduced in \cite{Kwietniak-Lacka} the FK-pseudometric for discrete-time systems. Independently, Ratner \cite{Ratner_2017} proposed a notion of matching between segments of orbits for flows that can be seen as a continuous-time version of the $\bar{f}_n$-metric with respect to partitions (see section \ref{LK_flows} for details).
	Motivated by these works we propose a definition of the Feldman-Katok pseudometric for flows.
	
	We start defining a topological matching based on the metric instead of using a partition. The definitions here will behave similarly to the ones presented previously for maps. To keep the intuition behind the definitions we will use similar notation using a tilde to emphasize the continuous-time setting.
	
	\begin{defi}
		Let $x,y \in X$, $t>1$, $0<\vep<1$ and $\de>0$. We say that $x$ and $y$ are \textbf{$\wt{(t, \vep,\de)}$-matchable} if there exist measurable sets $A, A' \subseteq [0,t]$ with $\la(A)>(1-\vep)t$, $\la(A') > (1-\vep)t$ and an increasing absolutely continuous onto map $h:A \to A'$  satisfying for all $s \in A$:
		\begin{enumerate}
			\item $|h'(s) - 1|< \vep $,
			
			\item $d(\varphi^s(x), \varphi^{h(s)} (y)) < \de$.
		\end{enumerate}
		We call $h$ a \textbf{$\wt{(t,\vep,\de)}$-matching} between $x$ and $y$.
	\end{defi}
	
	The natural way to define $\wt{(t,\de)}$-gap as the analogue of the $(n,\de)$-gap is the following.
	
	\begin{defi}
		The \textbf{$\wt{(t,\de)}$-gap} between $x$ and $y$ is given by
		$$\tilde{f}_{t,\de}(x,y) = \inf\{\vep >0 : x \text{ and } y \text{ are } \wt{(t, \vep,\de)}\text{-matchable} \}.$$
		We define $\tilde{f}_{t,\de}(x,y) = 1$ in case the set is empty.
	\end{defi}
	
	In \cite{Kwietniak-Lacka} facts about the $(n,\de)$-gaps are left to the reader. For flows these facts remain true and can be proved following the same ideas. For example, we have:
	
	\begin{fact}
		If $0<\de < \de'$, then $\tilde{f}_{t,\de'}(x,y) \leq $$\tilde{f}_{t,\de}(x,y)$.
	\end{fact}
	
	\begin{fact}
		If $s \geq 1$, then $\tilde{f}_{t+s,\de} (x,y) \leq \tilde{f}_{t,\de} (x,y) + \frac{s}{t}$.
	\end{fact}
	
	The next property is analogous to Fact 14 in \cite{Kwietniak-Lacka} which elucidates the first connection between the pseudometric and ergodic theory.
	Despite the simplicity of the proof, we decided to present it to give a flavor of why we have less control of the matching in our setting. The difference here is that we obtain the result for $2\vep$ instead of $\vep$ as in the discrete-time case.
	
	\begin{Prop}\label{FK_measure}
		If $\tilde{f}_{t,\de} (x,y) < \vep$, then $D_P(\mu_{x,t}, \mu_{y,t})< \max \{\de, 2\vep\}$.
	\end{Prop}
	
	\begin{proof}
		Let $E \subseteq X$ be a Borel set. We will show that $\mu_{x,t}(E) \leq \mu_{y,t}(E^\be) + \be$ where $\be =\max \{\de,2\vep\}$. For $z \in X$ and $B \subseteq X$ we denote $B(z,t) = \{s \in [0,t] : \smap(z) \in B\}$.
		
		Let $h: A \to A'$ be a $\wt{(t, \vep,\de)}$-matching between $x$ and $y$.
		Since $\la(A)>(1-\vep)t$ we have
		$$ \la(E(x,t)) \leq \la (E(x,t) \cap A) + \vep t.$$
		
		Since $|h'(s) - 1|< \vep $ for $s \in A$, we also have $$ \la(E(x,t) \cap A ) \leq \la(h(E(x,t) \cap A) )  + \vep t.$$
		
		If $s \in A \cap E(x,t)$, then $\hsmap(y) \in E^\de$, which means $h(s) \in E^\de(y,t)$. Thus $h(E(x,t) \cap A) \subseteq E^\de(y,t)$ and consequently, putting it all together we obtain
		$$ \la(E(x,t)) \leq \la(h(E(x,t) \cap A) ) + 2\vep t \leq \la(E^\de(y,t)) + 2\vep t \leq \la(E^\be(y,t)) + \be t.$$
		
		Therefore, we conclude that
		$$ \mu_{x,t}(E) = \frac{\la(E(x,t)) }{t} \leq \frac{\la(E^\be(y,t)) + \be t}{t} = \mu_{y,t}(E^\be) + \be. \qedhere$$
	\end{proof}
	
	\begin{defi}
		The \textbf{$\tilde{f}_\de$-pseudodistance} is defined as
		$$\tilde{f}_\de(x,y) =\limsup_{t \to \infty} \tilde{f}_{t,\de}(x,y).$$
	\end{defi}
	
	\begin{fact}
		For $\tilde{f}_\de$ the invariance along orbits obtained in the discrete-time case is preserved. So if $z \in \SO(x)$, then $\tilde{f}_\de(z,y) = \tilde{f}_\de(x,y)$.
	\end{fact}
	
	\begin{defi}
		The \textbf{Feldman-Katok pseudometric} (FK-pseudometric for short) is given by
		$$\fkflow(x,y) = \inf \{\de >0 : \tilde{f}_\de(x,y)<\de\}.$$
	\end{defi}
	
	As a direct consequence of Proposition \ref{FK_measure} we obtain the following measure-theoretical information from the FK-pseudometric.
	
	\begin{Prop}\label{generic_points}
		If $x \in \gen(\mu)$ and $\fkflow(x,y)=0$, then $y \in \gen(\mu)$.
	\end{Prop}
	
	The analogy between the FK-pseudometric $\fk$ for discrete-time systems and the FK-pseudometric $\fkflow$ for flows raises the natural question if there is a relation between $\fkflow(x,y)$ for a flow and $\fk(x,y)$ for its time-$t$ maps.
	
	In general, we can only guarantee that topological and ergodic properties of a flow are projected onto many but not all time-$t$ maps. This happens with properties like minimality, ergodicity and unique ergodicity. 
	
	\begin{Prop}[Proposition 4 in \cite{Fayad}]\label{timet_transitive_minimal}
		Let $\flowF$ be a continuous flow. If $\flowF$ is transitive (resp. minimal),  then $\tmap$ is transitive (resp. minimal) for a dense $G_\de$ set of $t \in \R$. 
	\end{Prop}
	
	A similar result in the ergodic setting was proved by Pugh and Shub.
	
	\begin{Prop}[Theorem 1 in \cite{PS}]
		Let $\flowF$ be a continuous flow on $X$ and $\mu \in \Mes(X)$. If $(X, \flowF, \mu)$ is ergodic, then $(X,\tmap, \mu)$ is ergodic for all but countably many $t \in \R$.
	\end{Prop}
	
	The proof (due to Veech) of Theorem 3.3.34 in \cite{Hasselblatt-Fisher} implies the next result.
	\begin{Prop}\label{Unique_ergodic_tmap}
		If a continuous flow $\flowF$ is uniquely ergodic, then $\tmap$ is uniquely ergodic for all but countably many $t \in \R$.
	\end{Prop}
	
	On the other hand, sometimes it is enough that a property holds for a time-$t$ map  for some $t$ to guarantee that it holds for the flow. Examples of the latter properties are ergodicity and unique ergodicity (see \cite{Hasselblatt-Fisher} for more details). 
	
	In both Theorem \ref{thmA} and Theorem \ref{thmB}, we focus on zero FK-distance between pairs of points. 
	We will show that if the FK-distance between a pair of points with respect to a time-$t$ map vanishes for some $t$, then the FK-distance between these points with respect to the flow also vanishes.	
	
	\begin{Prop}\label{timet_FK}
		Let $x,y \in X$ and $\flowF$ be a continuous flow on $X$. If $\fk(x,y) = 0$ for some time-$t$ map, then $\fkflow(x,y)=0$. 
	\end{Prop}
	
	\begin{proof}
		Observe that it is enough to prove the proposition only for the time-$1$ map. Take $x,y \in X$ with $\fk (x,y) = 0$ for the time-$1$ map and $\vep >0$. We will show that there exists $t_0>1$ such that $\tilde{f}_{t,\vep}(x,y) \leq \vep$ for all $t \geq t_0$. 
		
		By continuity of the flow there exists $0<\de<\frac{\vep}{2}$ such that $d(\smap(w), \smap(z)) <~ \vep$ for $0 \leq s \leq 1$ for any $w,z \in X$ with $d(w,z) < \de$.  
		
		Since $\fk (x,y) = 0$ for the time-$1$ map, we know there exists $n_0 \in \N$ such that $\bar{f}_{n,\de} (x,y) \leq \de$ for every $n \in \N$ with $n \geq n_0$.
		
		Let $t_0 >  n_0$ be such that $\frac{t_0 \vep}{4}>1$. Fix $t \geq t_0$. Then, there exists $n \in \N$ with $n \geq n_0$ such that $n < t \leq n+1$. Since $\bar{f}_{n+1,\de} (x,y) \leq \de$, there exists $\pi \colon D \to D'$ an $(n + 1,\de)$-matching between $x$ and $y$ with $|\pi| \geq (1-\de)(n+1)$.
		
		Let $D_0 = D \cap [0,n-1] \cap \pi^{-1} (D' \cap [0,n-1])$. Note that 
		\begin{enumerate}
			\item $|D_0| \geq |D| - 2$;
			
			\item $[k, k+1) \subset [0,t]$ for every $k \in D_0$;
			
			\item $[\pi(k), \pi(k)+1) \subset [0,t]$ for every $k \in D_0$.
		\end{enumerate}
		
		From (1) we have
		$$	|D_0| \geq (1-\de)(n+1) - 2   \geq (1-\de)t -2
		> (1-\vep)t.$$
		
		Set 
		$$
		A =\bigcup_{k \in D_0} [k, k+1) \text{ and }A'= \bigcup_{k \in D_0} [\pi(k), \pi(k)+1).
		$$
		
		Let $h\colon A \to A'$ be the isometric translation from $[k, k+1)$ to $[\pi(k), \pi(k) +1)$ for each $k \in D_0$.
		By construction we have $A, A' \subseteq [0,t]$ and $\la(A)= \la(A') = |D_0| > (1-\vep)t$. Since $d(\vphi^k(x), \vphi^{\pi(k)}(y))< \de$ for every $k \in D_0$, from the choice of $\de$ we conclude that $d(\smap(x), \hsmap(y))< \vep$ for every $s \in A$. Therefore, $h$ is a $\wt{(t,\vep,\vep)}$-matching between $x$ and $y$ which completes the proof. 
	\end{proof}
	
	The flexibility in the definition of matchings for flows suggests that zero FK-distance between points $x,y \in X$ for a flow does not imply vanishing of $\fk(x,y)$ for all time-$t$ maps. Indeed this is not true in general: we prove in Section \ref{applications} that there exist flows with a pair of points $x,y \in X$ such that $\fkflow(x,y)=0$ but with $\fk(x,y)>0$ for a time-$t$ map.
	
	Another natural question regarding the relations between the definition of the FK-pseudometric for flows and for maps concerns suspension flows. A direct consequence of the Proposition \ref{timet_FK} is the following.
	
	\begin{corollary}\label{suspension_LK}
		Let $T\colon X \to X$ be a homeomorphism, $x,y \in X$ and $\flowF$ be the suspension flow over $T$. If $\fk(x,y) =0$ with respect to $T$, then for any $r,s \in [0,1)$ we have $\fkflow((x,r), (y,s)) = 0$ with respect to $\flowF$.
	\end{corollary}
	
	\begin{proof}
		Let $x,y \in X$ be such that $\fk(x,y)=0$ with respect to $T$. For $z \in X$, denote $\wh{z} = (z,0) \in X_T$. Since $\fkflow$ is invariant along orbits we can assume, without loss of generality that  $r=s=0$. Now, it is enough to show that $\fkflow(\wh{x}, \wh{y})=0$.

		Let $\pfk$ denote the FK-pseudometric on $X_T$ with respect to $\vphi^1$. Note that $d_T(\wh{w}, \wh{z}) = d(w,z)$ for every $w,z \in X$ and that $\vphi^1(\wh{z})=\wh{T(z)}$ for every $z  \in X$. So $\fk(x,y)=0$ implies $\pfk(\wh{x}, \wh{y})=0$ and consequently, by Proposition \ref{timet_FK}, $\fkflow(\wh{x}, \wh{y})=0$.
	\end{proof}
	
	\section{Loosely Kronecker flows}\label{LK_flows}
	The goal of this section is to prove Theorem \ref{thmA} characterizing loosely Kronecker continuous flows. First, we need to recall some properties of the space of measurable partitions. 
	
	Whenever we mention a \textbf{partition} we mean a finite partition of $X$ into Borel sets called \textbf{cells}. We write $\Part(X)$ for the set of all partitions of $X$.
	
	To any ordering of a given partition $\SP = \{P_1, \ldots, P_n\}$ we can associate a function on $X$ defined by $\SP(x) = j$ for $x \in P_j$.	
	
	Let $\mu \in \SM(X)$. A pseudometric on $ \Part(X)$ is defined by
	$$d_\mu(\SP,\SQ) = \inf \left\{\frac{1}{2}\sum_{j=1}^{\max \{|\SP|,|
		\SQ|\}} \mu(P_j \triangle Q_j) \right\} =  \inf \left\{ \mu\left(\left\{x \in X: \SP(x) \neq \SQ(x)\right\} \right)\right\},$$
	where the infimum is taken over all possible orderings for the partitions $\SP$ and $\SQ$. In case $|\SP| \neq |\SQ|$ we add empty cells to the partition with fewer elements in order to have two partitions with the same cardinality. If we identify partitions $\SP,\SQ$ such that $d_\mu(\SP, \SQ)=0$, then $d_\mu$ is indeed a metric on $\Part(X)$ (here we abuse the terminology and treat equivalence classes of partitions as partitions). 
	
	Since our argument uses Ratner's criterion for loose Bernoullicity we recall her results from \cite{Ratner_2017}.
	
	\begin{defi}\label{matching_ratner}
		Let $x,y \in X$, $t>1$, $\SP$ be a partition of $X$ and $\vep >0$. We say that $x$ and $y$ are \textbf{$(t, \vep, \SP)$-matchable} if there exist measurable sets  $A,A' \subseteq [0,t]$ with $\la(A) > (1-\vep)t$, $\la(A') > (1-\vep)t$ and an increasing absolutely continuous onto map $h:A \to A'$ satisfying for all $s \in A$:
		\begin{enumerate}
			\item $|h'(s) - 1|< \vep $,
			
			\item $\SP(\varphi^s(x)) =  \SP(\varphi^{h(s)} (y)) .$
		\end{enumerate}
		We call $h$ a \textbf{$(t, \vep, \SP)$-matching} between $x$ and $y$.
	\end{defi}
	
	For $x,y \in X$, partition $\SP \in \Part(X)$, $\vep >0$, and $t >1$ we set:
	\begin{align*}
		&f_t(x,y,\SP) = \inf \{ \vep > 0 :  x \text{ and }y \text{ are } (t, \vep, \SP) \text{-matchable}\}, \\ 
		&B_t(x,\vep, \SP)= \{y \in X : f_t(x,y, \SP)< \vep\}.
	\end{align*}	
	
	We remark that $f_t( \cdot, \cdot,\SP)$ is not a metric since it does not satisfy the triangle inequality, however it is not far from being a metric as $z,y \in B_t(x,\vep, \SP)$ implies $f_t(z,y,\SP) \leq 5\vep$. 
	
	The set $B_t(x,\vep, \SP)$ is called the \textbf{$(t, \SP)$-ball} of radius $\vep>0$ centered in $x$.
	
	A family $\SA$ of $(t,\SP)$-balls of radius $\vep >0$ is called an \textbf{$(\vep, t, \SP)$-cover} of $X$ if $\mu(\cup_{A \in \SA} A) > 1-\vep$.
	
	Let $K_t(\vep,\SP) = \inf \{|\SA| : \SA \text{ is a an } (\vep,t,\SP)\text{-cover of } X\}$, where $|\SA|$ represents the cardinality of $\SA$.
	
	Let $U$ denote the family of all positive non-decreasing functions from $\R^+$ onto itself.  For $u \in U$ denote: $$\be(u, \vep, \SP) = \liminf_{t \to \infty} \frac{\log K_t(\vep, \SP)}{u(t)} ;$$
	$$e(u,\SP) = \limsup_{\vep \to 0} \be(u, \vep, \SP);$$
	$$ e(\flowF,u)  = \sup \{e(u,\SP) : \SP \in \Part(X)\}.$$
	
	In a previous work from 1981, Ratner \cite{Ratner_invariants} used a definition of matching that requires the $(t, \vep, \SP)$-matching to be an increasing function preserving the Lebesgue measure. Later, she proved that for ergodic flows the quantity $e(\flowF,u)$ remains the same if we use absolutely continuous or measure-preserving matching (Theorem 4 in \cite{Ratner_2017}), recovering a characterization of loosely Kronecker flows originally proved in 1981.
	
	\begin{theorem}[Ratner]\label{ratnerLK}
		An ergodic flow $\flowF$ is loosely Kronecker if and only if $e(\flowF,u)=0$ for all $u \in U$.
	\end{theorem}
	
	The proof of Theorem \ref{ratnerLK} in \cite{Ratner_invariants} relies on a characterization of loosely Kronecker flows due to Feldman \cite{Feldman_rentropy} which states that an ergodic flow is loosely Kronecker if and only if for any $\vep >0$ and $\SP \in \Part(X)$ we have $K_t(\vep, \SP) =1$ (with respect to measure-preserving matchings) for every $t$ large enough.
	
	Since the FK-pseudometric is a topological notion, we focus on partitions with good topological properties and we show that they are enough to prove $e(\flowF,u)=0$. In fact, we use the these partitions to show that for for any $\vep >0$ and $\SP \in \Part(X)$ we have $K_t(\vep, \SP) =1$ (with respect to absolutely continuous matchings) for every $t$ large enough.
	
	We start by fixing $\SU$ a countable basis for the topology of $X$. We can assume without loss of generality that $\SU$ is closed under finite unions and intersections. We say that a partition is \textbf{$\SU$-regular }if at most one cell does not belong to $\SU$. 
	
	\begin{defi}
		For a given $\mu \in \SM(X)$, we say that a partition $\SP \in \Part(X)$ is \textbf{$\vep$-essentially open} if it is $\SU$-regular and if there is a cell $P \in \SP$ that does not belong to $\SU$, then it satisfies $\mu(P)< \vep$. 
	\end{defi}
	
	Besides having good topological properties, for each $\mu \in \SM(X)$ the family of  $\vep$-essentially open partitions constitutes a dense set in the space of measurable partitions.
	
	\begin{Prop}\label{essent_open_part}
		Let $\mu \in \SM(X)$ and $\vep>0$. Then the set of $\vep$-essentially open partitions is dense in $\Part(X)$ with respect to $d_\mu$.
	\end{Prop}
	
	\begin{proof}
		Let $\SP=\{P_1,\ldots, P_n\} \in \Part(X)$. Fix $0< \de < \vep$.  By regularity of $\mu$, for each $i \in \{1,\ldots,n\}$ there exists a compact $K_i \subseteq P_i$  such that $\mu(P_i\backslash K_i)< \frac{\de}{4n}$. Take $\{Q_1, \ldots, Q_n \} \subseteq \SU$ disjoint such that $K_i \subseteq Q_i$ and $\mu(Q_i \backslash K_i)< \frac{\de}{4n}$ for $i \in \{1,\ldots,n\}$. Set $Q_{n+1} = X \backslash \cup_{i=1}^n Q_i$ and $\SQ = \{Q_1,\ldots,Q_{n+1}\}$.
		
		Observe that since $Q_{n+1} \subset \cup_{i=1}^n P_i\backslash K_i$ we have $\mu(Q_{n+1})< \frac{\de}{4}< \vep$ which means that $\SQ$ is $\vep$-essentially open. Moreover, $\mu(P_j \triangle Q_j)<\frac{\de}{2n} $ for every $j \in \{1,\ldots,n\}$ and therefore  $d_\mu(\SP, \SQ)<\de$ which completes the proof.
	\end{proof}
	
	One direction of Theorem \ref{thmA} is a direct consequence of Garc\'ia-Ramos and Kwietniak's result (Theorem 4.5 in \cite{KG-R}) together with some properties of the FK-pseudometric already presented. 
	
	For the other implication in Theorem \ref{thmA}, we use Ratner's criterion. We will also show that it is possible to estimate $e(u, \SP)$ for any partition $\SP$ by $e(u, \SQ)$ where $\SQ$ is a partition close to $\SP$. Therefore the density of essentially open partitions will allow us to focus on them in order to compute $e(u, \SP)$ for any partition $\SP \in \Part(X)$.
	
	\begin{lemma}\label{key_lemma}
		Let $\flowF$ be a continuous flow on $X$, $\mu \in \Erg(X)$, and $\SP,\SQ \in \Part(X)$. For every $\de>0$ there exist $H \subset \gen(\mu)$ and $t_0=t_0(\SP,\SQ)>1$ such that $\mu(H) > 1-\de$ and for every $x,y \in H$ and $t \geq t_0$ we have $$f_t(x,y,\SP) \leq 4 f_t(x,y,\SQ) + 2 d_\mu(\SP,\SQ) + \de.$$
	\end{lemma}
	
	\begin{proof}
		Assume that $\SP$ and $\SQ$ are ordered optimally with respect to $d_\mu$. Let $D= \{z \in X : \SP(z) \neq \SQ(z)\}$ and $C = X \backslash D$. For any $z \in X$ and $t \in \R^+$ we will denote by $C(z,t) = \{s \in [0,t] : \smap(z) \in C\}$. Since $\mu$ is ergodic, applying the Birkhoff Ergodic Theorem to the characteristic function on $D$, we obtain that for $\mu$-almost every $z \in X$ it holds
		$$\lim_{t \to \infty} \mu_{z,t}(D) = \lim_{t \to \infty}\int_{0}^{t} \chi_D(\smap(z)) =\int_X \chi_D d\mu= \mu(D) .$$ So by Egorov's Theorem there exists $H \subseteq \gen(\mu)$ and $t_0>0$ such that $\mu(H) > 1-\de$ and for every $z \in H$ and $t\geq t_0$ we have
		\begin{equation}\label{unif_conv}
			| \mu_{z,t}(D)- \mu(D)| = |\mu_{z,t}(D)- d_{\mu}(\SP, \SQ)|  <\frac{\de}{2}.
		\end{equation}
		
		We will show that if $x,y \in H$ are $(t,\vep, \SQ)$-matchable, then they also are $(t, 4 \vep + 2d_\mu(\SP,\SQ) + \de, \SP)$-matchable. Let $x,y \in H$ and $h: A_0 \to A_0'$ be a $(t,\vep, \SQ)$-matching between $x$ and $y$.
		
		We will consider the map $h|_A: A \to h(A)$ where $A$ is defined as$$A = A_0 \cap C(x,t)\cap h^{-1}(A_0' \cap C(y,t)).$$
		From (\ref{unif_conv}) we have 
		\begin{equation}\label{eq_Cxt}
			\frac{\la(C(x,t))}{t} > 1- d_\mu(\SP,\SQ) - \frac{\de}{2},
		\end{equation}
		and the same holds for $\la(C(y,t))$. Using these estimates together with the control of the measure by the derivative of $h$ (see condition (1) in Definition \ref{matching_ratner}) and the fact that $\frac{\la(A_0')}{t} > 1 - \vep$ we have \begin{equation}\label{eq_sizeA0}
			\frac{\la(h^{-1}(A_0' \cap C(y,t)))}{t} \geq \frac{\la(A_0' \cap C(y,t))}{t} (1-\vep)> 1- 2\vep - d_\mu(\SP,\SQ) - \frac{\de}{2}.
		\end{equation}
		Thus putting \eqref{eq_Cxt} and \eqref{eq_sizeA0} together we obtain
		\begin{equation}
			\frac{\la(A)}{t} > 1 - 3 \vep - 2d_\mu(\SP,\SQ) - \de.
		\end{equation}
		Using again the control of the measure by the derivative of $h$  and (4.4) we have
		$$\frac{\la(h(A)) }{t} \geq \frac{\la(A)}{t}(1-\vep)> 1 - 4 \vep - 2d_\mu(\SP,\SQ) - \de.$$
		
		By construction, for every $s \in A$ we have $\smap(x) \notin D$ and $\hsmap(y) \notin D$ which implies $$\SP(\smap(x)) = \SQ(\smap(x)) = \SQ(\hsmap(y)) = \SP(\hsmap(y)).$$ Therefore, $h|_A$ is a $(t, 4 \vep + 2d_\mu(\SP,\SQ) + \de, \SP)$-matching between $x$ and $y$.
	\end{proof}
	
	The last ingredient is due to Ornstein, Rudolph and Weiss (see Chapter 7 in \cite{ORW}). 
	The result was first obtained by Rudolph in \cite{Rudolph_timet} for loosely Bernoulli flows (the positive entropy case) and a modification of the argument led to an analogous result for loosely Kronecker flows.
	
	\begin{theorem}\label{timet_LK}
		Every ergodic time-$t$ map of a loosely Kronecker flow is loosely Kronecker.
	\end{theorem}
	
	Now that we have all the ingredients, we recall Theorem \ref{thmA} and present the proof.
	
	\begin{Thm}
		Let $\flowF$ be a continuous flow on $X$ and $\mu \in \Erg(X)$. The measure-preserving flow $(X, \flowF, \mu)$ is loosely Kronecker if and only if there exists a Borel set $H \subseteq X$ such that $\mu(H) = 1$ and $\fkflow(x,y) = 0$ for every $x,y \in H$.
	\end{Thm}

	\begin{proof}
		Fix an ergodic time-$t$ map $\tmap$ of a loosely Kronecker flow $(X, \flowF, \mu)$. It follows from Theorem \ref{timet_LK} that $\tmap$ is a loosely Kronecker transformation. 
		
		By the discrete-time result analogous to Theorem \ref{thmA} (Theorem 4.5 in \cite{KG-R}) we obtain a Borel set $H \subseteq X$ with $\mu(H)=1$ such that $\fk(x,y)=0$ with respect to $\tmap$ for every $x,y \in H$ . By Proposition \ref{timet_FK} we obtain $\fkflow(x,y) = 0$ for every $x,y \in H$.
		
		For the converse we fix $\vep >0$, a partition $\SP \in \Part(X)$ and a Borel set $H_0\subseteq X$ such that $\mu(H_0) = 1$ and $\fkflow(x,y)=0$ for every $x,y \in H_0$. 
		
		From the density of essentially open partitions there exists an $\vep$-essentially open partition $\SQ=\{Q_0,\ldots, Q_{n}\}$ such that $d_\mu(\SP, \SQ)<\vep$. 
		We assume without loss of generality that $Q_0$ is the compact cell and take an open set $V_0$ containing $Q_0$ such that $\mu(V_0)< \vep$.
		
		Let $0<\de<\vep$ be a Lebesgue number for the open cover $\SW = \{V_0, Q_1,\ldots, Q_{n}\}$ of $X$.
		
		Since $\mu$ is ergodic, there exist a compact set $H_1 \subseteq \gen(\mu) \cap H_0$ and $t_1 >1$ such that $\mu(H_1)> 1-\de$ and $\mu_{x,t}(V_0)< \vep$ for every $t \geq t_1$ and $x \in H_1$.
		
		Fix $x_0 \in H_1$. We observe that $\tilde{f}_{t,\de}(x_0, \cdot)$ is continuous for every $t \geq 1$, so $\tilde{f}_\de(x_0, \cdot)$ is measurable. Thus there exist a compact set $H_2 \subseteq H_1$ and $t_2 \geq t_1$ such that $\mu(H_2)>1-\de $ and $\tilde{f}_{t,\de}(x_0,y)<\de$ for every $t \geq t_2$ and $y \in H_2$.
		
		For $t \geq t_2$ and $y \in H_2$ let $h \colon A_0 \to A_0'$ be a $\wt{(t,\de,\de)}$-matching between $x_0$ and $y$. Set $A \subseteq A_0$ as
		$$A = \{s \in A_0 : \smap(x_0) \notin V_0\}.$$
		
		Since all elements of the cover $\SW$, except possibly $V_0$, are cells of the partition $\SQ$ and $\de$ is a Lebesgue number for $\SW$, for every $s \in A$ we have $\SQ(\smap(x_0)) = \SQ(\hsmap(y))$.
		
		Using an argument similar to the one in the proof of Lemma \ref{key_lemma} together with the fact that $\de < \vep$ we can see that $h|_A  \colon A \to h(A)$ is a $(t, 3\vep, \SQ)$-matching between $x_0$ and~$y$.
		
		Lastly, we apply Lemma \ref{key_lemma} and assume that the obtained set $H$ is contained in $H_2$ and $t_0 \geq t_2$. So, for every $y \in H$ and $t \geq t_0$ we have
		$$f_t(x_0,y,\SP)<4f_t(x_0,y,\SQ) + 2d_\mu(\SP,\SQ) + \de < 15 \vep.$$
		
		This means that $H \subseteq B_t(x_0,15\vep, \SP)$, and consequently, $K_t(15\vep,\SP) = 1$ for every $t \geq t_0$. Since $\SP \in \Part(X)$ and $\vep>0$ were arbitrary, we have $e(\flowF,u) = 0$ for every $u \in U$. Using Theorem \ref{ratnerLK} we conclude that $(X, \flowF, \mu)$ is loosely Kronecker.
	\end{proof}
	
	\section{Topologically loosely Kronecker flows}\label{topologically_loosely_Kronecker_section}
	
	The goal of this section is to prove Theorem \ref{thmB}. We begin defining topologically loosely Kronecker flows in analogy with the notion introduced for discrete-time systems in \cite{KG-R}.
	
	\begin{defi}
		A continuous flow is said to be \textbf{topologically loosely Kronecker} if $\fkflow(x,y) = 0$ for every $x,y \in X$.
	\end{defi}
	
	The following lemma is a direct consequence of Proposition \ref{generic_points}.
	
	\begin{lemma}\label{TLK_unique}
		Every topologically loosely Kronecker flow is uniquely ergodic.
	\end{lemma}
	
	Using results presented above we can restate and prove Theorem \ref{thmB}.
	
	\begin{Thm}
		A continuous flow $\flowF$ on $X$ is topologically loosely Kronecker if and only if $\flowF$ is uniquely ergodic and $(X,\flowF, \mu)$ is loosely Kronecker, where $\{\mu\}~=~ \Erg(X)$.
	\end{Thm}
	
	\begin{proof} 
		If $\flowF$ is topologically loosely Kronecker, then, by Lemma \ref{TLK_unique}, it must be uniquely ergodic. Let $\{\mu\} = \Erg(X)$. Then $(X,\flowF, \mu)$ must be loosely Kronecker by Theorem \ref{thmA}.
		
		Conversely, if $\flowF$ is uniquely ergodic and $(X, \flowF, \mu)$ is loosely Kronecker with $\mu$ being its unique ergodic measure, then we can take a uniquely ergodic time-$t$ map whose existence is guaranteed by the Proposition \ref{Unique_ergodic_tmap} and, by Theorem \ref{timet_LK}, the measure preserving system $(X, \tmap, \mu)$ is loosely Kronecker.
		
		By the discrete-time result analogous to Theorem \ref{thmB} (Theorem 5.7 in \cite{KG-R}), the map $\tmap$ is topologically loosely Kronecker, that is, $\fk(x,y) = 0$ for every $x,y \in X$. Using Proposition \ref{timet_FK} we obtain $\fkflow(x,y)=0$ for every $x,y \in X$.
	\end{proof}	
	
	\section{Examples and applications}\label{applications}
	
	Horocycle flows on surfaces of constant negative curvature are probably the most studied examples of topologically loosely Kronecker flows (see \cite{furstenberg_uniquely_ergodic}, \cite{Marcus_horocycle_erg_prop},  \cite{Marcus_horocycle_mixing},  \cite{Ratner_horocycle}, \cite{Ratner_Cartesian} and others). There are some other famous classes, such as strictly ergodic distal flows, that are also known to be quasi-isometric (see \cite{Furstemberg_distal}). Since the class of loosely Kronecker flows is closed under isometric extensions (see Chapter 7 in \cite{ORW}), strictly ergodic distal flows must be loosely Kronecker with respect to its unique ergodic measure. 
	
	The measure-theoretical theory of discrete-time loosely Kronecker systems suggests that suspension flows are a natural source of examples for topologically loosely Kronecker flows. This is actually true and it follows from Corollary \ref{suspension_LK}.
	
	\begin{Prop}
		Suspension flows over topologically loosely Kronecker homeomorphisms are topologically loosely Kronecker.
	\end{Prop}
	
	\begin{proof}
		If $T:X \to X$ is topologically loosely Kronecker, then $\fk(x,y) =0$ for every $x,y \in X$. By Corollary \ref{suspension_LK} we have $\fkflow((x,r),(y,s))=0$ for every $(x,r) , (y,s) \in~X_T$. 
	\end{proof}
	
	The following corollary is a direct consequence of Kakutani equivalence of special flows over the same automorphism(Proposition 2.2 in \cite{Katok_discrete_spec}). We just need to assure that the special flow is indeed a continuous flow.
	
	\begin{corollary}
		Special flows over topologically loosely Kronecker homeomorphisms are topologically loosely Kronecker.
	\end{corollary}
	
	A natural question is whether dynamical properties persist after time-changes. While some properties such as topological mixing might not be preserved, others as transitivity and minimality are preserved. Being topologically loosely Kronecker is in the latter class and as expected from the definition of Kakutani equivalence it remains preserved by continuous time-changes.
	
	\begin{Prop}
		Continuous time-changes of topologically loosely Kronecker flows are topologically loosely Kronecker.
	\end{Prop}
	
	\begin{proof}
		We remark that unique ergodicity is preserved under continuous time-changes for non-singular flows (Corollary 2 in \cite{Marcus_uniquely_ergodic}). Moreover, a continuous time-change of a flow is obviously Kakutani equivalent to itself, so the obtained flow is uniquely ergodic and loosely Kronecker with respect to its unique measure. Therefore, by Theorem \ref{thmB}, it is topologically loosely Kronecker.
	\end{proof}
	
	In Section 3 we mentioned that vanishing of the FK-distance for flows might not be preserved after passing to time-$t$ maps. The measure-theoretical weak mixing property is used to prove that.
	
	\begin{defi}
		We say that a measure-preserving flow $(X, \flowF, \mu)$ is \textbf{weakly mixing} if $\flowF \times \flowF$ is ergodic.
	\end{defi}
	
	Weak mixing has some equivalent forms, one of them (also used by Ratner in \cite{Ratner_invariants}) is the following.
	
	\begin{Prop}[Proposition 3.4.40 in \cite{Hasselblatt-Fisher}]
		Let $\flowF$ be a continuous flow and $\mu \in \Erg(X)$. Then $(X, \flowF, \mu)$ is weakly mixing if and only if every time-$t$ map is ergodic.
	\end{Prop}
	
	So from Theorem \ref{timet_LK} we obtain the following corollary.
	
	\begin{corollary}
		Every time-$t$ map for $t \neq 0$ of a weakly mixing loosely Kronecker continuous flow is loosely Kronecker. 
	\end{corollary}
	
	We observe that the class of loosely Kronecker flows is not contained in the class of weakly mixing flows. For example, the suspension flow over an irrational rotation on the circle is loosely Kronecker by definition, and it is not weakly mixing (Proposition 3.4.9 in \cite{Hasselblatt-Fisher}). Since loosely Kronecker transformations are ergodic, another consequence of Theorem \ref{timet_LK} is the following.
	
	\begin{corollary}
		If a loosely Kronecker flow is not weakly mixing, then it has a non-loosely Kronecker time-$t$ map. In particular, there are topologically loosely Kronecker flows with non topologically loosely Kronecker time-$t$ maps.
	\end{corollary}
	
	\ack{The author would like to thank Philipp Kunde and Dominik Kwietniak for the fruitful conversations and the reviewers for the valuable comments that improved this work.}
	
	\bibliographystyle{alpha}
	\bibliography{references}
	
\end{document}